\documentclass[11pt, leqno]{article}
\usepackage[utf8]{inputenc}
\usepackage[english]{babel}
\usepackage{amsmath, amsthm,amsfonts,amssymb}
\numberwithin{equation}{section}
\usepackage{dsfont}
\usepackage{setspace}
\usepackage{enumerate}
\usepackage{faktor}
\usepackage[left=2.5cm,right=2.5cm,top=2.2cm,bottom=2.2cm, includefoot]{geometry}
\usepackage{ulem}
\usepackage{color}
\newcommand{\R}{\mathbb{R}} 
\newcommand{\N}{\mathbb{N}} 

\renewcommand{\phi}{\varphi}
\providecommand{\dx}{\:\mathrm{d}x}

\newcommand{\cC}{{\mathcal C}}

\newcommand{\cK}{{\mathcal K}}

\newcommand{\cR}{{\mathcal R}}

\newcommand{\bR}{\textbf{R}}

\newcommand{\weakto}{\rightharpoonup}

\newcommand{\1}{\mathds{1}}

\newcommand{\RN}{\R^{N}}
\newcommand{\ri}{\text{i}}

\newcommand{\SR}{\mathcal{S}}                      

\DeclareMathOperator*{\diam}{diam}
\DeclareMathOperator*{\dist}{dist}
\DeclareMathOperator{\supp}{supp}
\DeclareMathOperator{\esssup}{esssup}

\newtheorem{theorem}{Theorem}[section] 
\newtheorem{defi}[theorem]{Definition}

\newtheorem{proposition}[theorem]{Proposition}
\newtheorem{corr}[theorem]{Corollary}
\author{}
\date{}
\begin{document}
\title{Dual variational methods for a nonlinear Helmholtz equation with sign-changing nonlinearity}

\author{Rainer Mandel, Dominic Scheider, Tolga Ye\c{s}il}
\maketitle
\allowdisplaybreaks

\begin{abstract}
  We prove new existence results for a Nonlinear Helmholtz equation with sign-changing nonlinearity of the
  form 
  $$ - \Delta u - k^{2}u = Q(x)|u|^{p-2}u,  \quad u \in W^{2,p}(\RN) $$ with 
   $k>0,$  $N \geq 3$, $p \in \left[\left.\frac{2(N+1)}{N-1},\frac{2N}{N-2}\right)\right.$ and $Q \in
L^{\infty}(\RN)$. 
  Due to the sign-changes of $Q$, our solutions have infinite Morse-Index in the
  corresponding dual variational formulation.
  \end{abstract}

\section{Introduction}
In the present article, we consider   nonlinear Helmholtz equations of the form
 \begin{equation}\label{eqn:introQ}
 - \Delta u - k^{2}u = Q(x)|u|^{p-2}u \qquad \text{on } \RN
 \end{equation}
for $p \in \left[\left.\frac{2(N+1)}{N-1},\frac{2N}{N-2}\right)\right.$ and $k >0 $ with a weight function $Q \in
L^{\infty}(\RN)$ that may change sign. To allow for the latter is nontrivial given that one of the
main tools for proving the existence of solutions  is the dual variational method that, in its classical
form, relies on the nonnegativity of the weight function. In the context of Nonlinear Helmholtz equations it 
was first implemented in a paper by Ev\'{e}quoz and Weth~\cite{Evequoz2015}. 
To highlight the role of the
nonnegativity of $Q$ we briefly recapitulate the approach. 

\medskip

Instead of~\eqref{eqn:introQ} one considers a reformulation as the integral equation 
\begin{equation}\label{eqn:introintegraleq}
u = \bR(Q|u|^{p-2}u) \qquad u \in L^{p}(\RN),
\end{equation}
where $\bR$ is the real part of a resolvent type operator $\cR $, i.e.,  a right inverse  of the
Helmholtz operator $-\Delta-k^2$ on $\R^N$. For $f \in \SR(\RN)$ the operator $\cR $ is given by $\cR(f) = 
\Phi \ast f$ where 

$$ 
  \Phi(x):= \frac{\ri}{4}\left( \frac{k}{2\pi |x|}\right)^{\frac{N-2}{2}}H^{(1)}_{\frac{N-2}{2}}(k|x|), \quad
  x \in \RN\setminus\{0\} 
$$
is the fundamental solution of the Helmholtz equation associated with \textit{Sommerfeld's outgoing radiation
condition}\begin{equation}\label{eqn:introsommerfeld}
\left|\nabla \Phi(x) -  k\,\ri \Phi(x)\frac{x}{|x|}\right| = o(|x|^{\frac{1-N}{2}}), \quad \text{ as }|x|
\to \infty.
\end{equation}

Here, $H^{(1)}_{\frac{N-2}{2}}$ denotes the Hankel function of the first kind and order $\frac{N-2}{2}$. 
So the  operator $\bR$ from~\eqref{eqn:introintegraleq} is given by $\bR(f) = \Psi \ast f$ where
$\Psi := \text{Re}(\Phi)$ is given by 
\begin{equation}\label{eq:defnPsi}
  \Psi(x)= -\frac14 \left( \frac{k}{2\pi |x|}\right)^{\frac{N-2}{2}}Y_{\frac{2-N}{2}}(k|x|), \quad x \in
  \RN\setminus\{0\}
\end{equation} 
It is
known \cite[Theorem 2.3]{kenig1987} that  $\mathcal R$ extends as a continuous linear map from
$L^{p'}(\RN) \to L^{p}(\RN)$ precisely for $p \in \left[ \frac{2(N+1)}{N-1},\frac{2N}{N-2} \right]$.  
One then introduces the  dual variable $\tilde u:= Q^{1/p'}|u|^{p-2}u$ and observes that solutions of
\eqref{eqn:introintegraleq} are precisely the critical points of the (dual) energy functional $I: L^{p'}(\RN)
\to L^{p}(\RN)$ given by 
$$ 
  I(\tilde u):= \frac{1}{p'}\left\|\tilde u\right\|^{p'}_{p'} - 
  \frac12 \int\limits_{\RN} \tilde u\cK \tilde u~dx. 
$$
Here, $\cK : L^{p'}(\RN) \to L^{p}(\RN), \tilde u \mapsto  
Q^{\frac{1}{p}}\textbf{R}(Q^{\frac{1}{p}}\tilde u)$ is a symmetric operator in the sense of
\begin{equation}\label{eqn:symmetry}
\int\limits_{\RN} f \,\cK g ~dx = \int\limits_{\RN} g\, \cK f ~dx 
\qquad\text{for all } f,g \in L^{p'}(\RN), 
\end{equation}
 see~\cite[Lemma 4.1]{Evequoz2015}. Under the additional assumption that $Q$ vanishes at
 infinity, one obtains that $I$ is an odd  functional of class $\cC^1$ that has the Mountain Pass Geometry and
 satisfies the Palais-Smale Condition. So the existence of an unbounded sequence of solutions
 to~\eqref{eqn:introintegraleq} follows from the Symmetric Mountain Pass Theorem. 
 Inverting the transformation $u\mapsto \tilde u$ one thus obtains an unbounded sequence of solutions to the 
 nonlinear Helmholtz equation~\eqref{eqn:introQ}. This is the strategy proposed by Ev\'{e}quoz and
 Weth~\cite{Evequoz2015} for the focusing nonlinear Helmholtz equation $Q\geq 0$. We refer
 to~\cite{MaMoPe_Osc} for analogous results in the defocusing case $Q \leq 0$, where the dual variational
 approach was implemented for the dual variable $\tilde u:= |Q|^{1/p'}|u|^{p-2}u$. In view of these two
 results it is natural to ask for a dual variational approach work in the intermediate case of sign-changing
 $Q$. In this paper, we provide a solution for this problem.

\medskip

To treat sign-changing coefficients $Q \in L^{\infty}(\RN)$ we have to come up with a new idea to make the
dual variation approach work. We write $Q = Q_{+} - Q_{-}$ where $Q_{\pm}:= |Q| \1_{A_{\pm}}$ and
\begin{equation}\label{eqn:Apm}
  A_{+}:= \{Q > 0\} , \qquad A_{-}:=\{Q\leq 0\}.
\end{equation}
In fact we will consider $Q_{\lambda}:= \lambda Q_{+} -Q_{-}$ for $\lambda > 0$ in the following. Our main
idea is to introduce a tuple of dual variables $(\phi,\psi)\in L^{p'}(A_+)\times L^{p'}(A_-)$ associated with
$(u|_{A_+},u|_{A_-})$ and to derive a coupled system of nonlinear integral equations the solutions of which
are precisely the critical points of an associated strongly indefinite dual energy
functional. We will see that the indefiniteness comes from the presence of $Q_-$ and thus vanishes in the case 
of a nonnegative function $Q\geq 0$. In particular, the critical points of this dual energy functional 
will have infinite Morse index, which clearly distinguishes these solutions from the dual bound and ground
states obtained in~\cite{Evequoz2015} in the case $Q\geq 0$. We will explain the dual variational framework in detail in
Section~\ref{sec:DualVarFor}.
Our conditions for the existence of critical points involve  the linear operator
$\cK:L^{p'}(\RN)\to L^{p}(\RN),  f \mapsto |Q|^{\frac{1}{p}}\bR(|Q|^{\frac{1}{p}}f)$ as well as the numbers
\begin{equation}\label{eqn:introalphabeta}
  \alpha:= \max\limits_{\left\|\phi\right\|_{p'}  = 1, \atop \supp(\phi)\subset A_+}\int\limits_{\RN} \phi
  \cK \phi ~dx, \qquad \beta:= \max\limits_{\left\|\phi\right\|_{p'} 
  = \left\|\psi\right\|_{p'} =1, \atop \supp(\phi)\subset A_+,\supp(\psi)\subset A_-} \int\limits_{\RN}
  \phi \cK \psi ~dx.
\end{equation}
Since we will assume $\cK$ to be compact, both values are indeed attained. 
Moreover we have $\beta \geq 0$
and \cite[Lemma 4.2(ii)]{Evequoz2015} 
gives $\alpha >0$ once we assume that $A^+$ has positive measure, i.e., $Q^+\not\equiv 0$. Our
main result reads as follows.



\begin{theorem}\label{thm:introthm}
 Let $p \in  [\frac{2(N+1)}{N-1}, \frac{2N}{N-2})$ and $ Q\in
  L^\infty(\RN)$, $Q^+\not\equiv 0$. Moreover assume that  
  \begin{equation}\label{eq:abstractassumptions}
    \cK : L^{p'}(\RN) \to L^{p}(\RN)\text{ is compact  and } 
    \int\limits_{\RN}\psi \cK \psi~dx \geq  0 \text{ for all }\psi \in L^{p'}(A_{-}).
  \end{equation}
  Then for almost all $\lambda>\lambda_{0}:=(2\beta\alpha^{-1})^{p}$ there is a nontrivial  strong solution $u \in W^{2,q}(\RN)\cap
  \cC^{1,\gamma}(\RN)$ for all $q\in \left(\frac{2N}{N-1},\infty\right)$ and $\gamma \in (0,1)$ of
  \begin{equation}\label{eqn:introeq}
 - \Delta u - k^2 u = Q_{\lambda}(x)|u|^{p-2}u \qquad \text{on } \RN.
 \end{equation}
\end{theorem}


The proof relies on a combination of a saddle-point reduction and the abstract monotonicity trick by
Jeanjean-Toland \cite{jeantoland98}, which provides bounded Palais-Smale sequences (only) for almost all $\lambda>\lambda_0$. It would clearly be desirable
to extend our result to all $\lambda>\lambda_0$, but related a priori bounds seem to be out of reach. Notice also that
\cite[Theorem~1.4]{MaMoPe_Osc} suggests the existence of nontrivial solutions also for small $\lambda>0$,
possibly assuming the set $\{Q>0\}$ to be small enough and following a different variational approach.
Let us point out that $\lambda_0$ is small provided that the number $\beta$, which is the same as the
operator norm $\|\mathds{1}_{A_{+}}\cK(\1_{A_{-}})\|_{p'\to p}$, is small compared to $\alpha>0$. 
In the case $p>\frac{2(N+1)}{N-1}$ this can be achieved by considering coefficient functions $Q$ such that
$\dist(A_-,A_+)$ is large enough, see \cite[Lemma 2.6]{evequoz2020dual}.

\medskip
 

In the following Corollary, we show that the abstract conditions~\eqref{eq:abstractassumptions} hold for a
reasonable class of sign-changing functions $Q\in L^\infty(\RN)$. If for instance $Q$ vanishes at infinity,
then \cite[Lemma 4.1(ii)]{Evequoz2015} applied to $|Q|$ implies that $\cK: L^{p'}(\RN) \to L^{p}(\RN)$ is
compact. It is less immediate to verify the non-negativity assumption on the bilinear form 
\eqref{eq:abstractassumptions}.
From~\cite[Corollary~5.4]{chen2019complex} we infer that this condition holds for measurable sets $A_-$ with
small enough diameter. To be more precise, if $y_{\frac{N-2}{2}}$ denotes the first (positive) zero of the
Bessel function $Y_{\frac{N-2}{2}}$, then the
condition $\diam(A_{-}) \leq k^{-1}y_{\frac{N-2}{2}}$ is sufficient. To put this condition into
perspective, note that for $N=3$ we have $Y_{\frac12}(t) = -\sqrt{\frac{2}{\pi t}}\cos{t}$, thus $y_{1/2} = \pi/2$ and $y_{\frac{N-2}{2}} >
y_{1/2}$ for $N >3$ (see \cite[Section 9.5]{abrastegun}).
We thus conclude as follows.

\begin{corr}\label{cor}
 Assume $p \in  [\frac{2(N+1)}{N-1}, \frac{2N}{N-2})$ and $Q\in
  L^\infty(\RN),Q^+\not\equiv 0$ . Moreover assume    
  
  \begin{equation}\label{eq:SufficientConditions}
	 \lim_{R\to\infty} \esssup_{|x|\geq R} |Q(x)| = 0\qquad\text{and}\qquad
    \diam{(A_{-})} \leq k^{-1} y_{\frac{N-2}{2}}.  
  \end{equation}

  Then for almost all $\lambda>\lambda_0:=(2\beta\alpha^{-1})^{p}$  there is a nontrivial  strong
  solution $u \in W^{2,q}(\RN)\cap \cC^{1,\gamma}(\RN)$ for all $q\in \left(\frac{2N}{N-1},\infty\right)$ and
  $\gamma \in (0,1)$ of \eqref{eqn:introeq}.
\end{corr}

The regularity results in Theorem \ref{thm:introthm} and Corollary \ref{cor} are direct
consequences of \cite[Lemma 4.3]{Evequoz2015} and of the iteration procedure from Step~3 in the proof of
\cite[Theorem~1]{Man_Uncountably}. Notice that \cite[Theorem~1]{Man_Uncountably} provides solutions
to far more general Nonlinear Helmholtz equations than~\eqref{eqn:introQ} regardless of whether sign-changes occur or not,
but the constructed solutions are small. 
This result relies on a fixed point approach. Let us also mention~\cite{EvWe_Real} where nontrivial solutions
of Nonlinear Helmholtz equations are constructed for rather general and possibly sign-changing nonlinearities
vanishing identically outside some compact subset of~$\RN$. Our method is entirely different from any of
these approaches.

\medskip

This paper is organzied as follows: In Section 2 we introduce our basic tools and develop the dual
variational framework by reformulating the indefinite Nonlinear Helmholtz equation as a coupled system of
integral equations. Then we prove that nontrivial critical points of the associated energy functional
$J_{\lambda}$ are indeed nontrivial solutions $u \in L^{p}(\RN)$ of the integral equation 
$u  = \bR(Q_{\lambda}|u|^{p-2}u)$. This motivates the search for critical points of $J_\lambda$.  
In Section~3 we perform the saddle-point reduction of $(\phi,\psi)\mapsto J_\lambda(\phi,\psi)$ leading to a
reduced function $\tilde J_\lambda$ that depends on $\phi$ only. In Section  4 we establish the existence of
bounded Palais-Smale sequences for these reduced functionals for almost all $\lambda > \lambda_{0}$. As
mentioned above, this step entirely relies on the monotonicity trick  by Jeanjean and
Toland~\cite{jeantoland98}. Finally, we combine all the auxiliary results to prove Theorem \ref{thm:introthm}
and Corollary \ref{cor} in Section 5. 

\medskip

Let us close this introduction by fixing some notation: Throughout the paper we denote by $B_{r}(x)$ the open
ball in $\RN$ with radius $r>0$ and center at $x \in \RN$.
Moreover, we set $B_{r} = B_{r}(0)$ for any $r>0$. For $1\leq s\leq \infty$, we abbreviate the standard norm
on $L^{s}(\RN)$ by $\left\|\cdot\right\|_{s}$. The Schwartz-class of rapidly decreasing functions on $\RN$ is
denoted by $\SR(\RN)$. For any $p \in (1,\infty)$ we always denote by $p':=\frac{p}{p-1}$ the H\"older
conjugate of $p$. The indicator function of a measurable set $B\subset\RN$ is $\1_B$. By $\diam{(B)}$ we
always denote the diametere of a set. We will always use the symbols
$\phi,\psi$ to denote $L^{p'}(A_{+})-$ and $L^{p'}(A_{-})$-functions that are continued trivially to the whole
of~$\RN$.
 
\medskip

\textbf{Acknowledgements}\\
The authors would like to thank Tobias Weth for helpful suggestions and stimulating discussions.
The first two authors are funded by the Deutsche Forschungsgemeinschaft (DFG, German Research Foundation)
- Project-ID 258734477 - SFB 1173.

\section{Dual variational formulation} \label{sec:DualVarFor}

In this section we will formulate a variational framework to the equation~\eqref{eqn:introeq}. 
We recall from the introduction that solutions of our problem are obtained as solutions of the integral
equation 
\begin{equation}\label{eqn:integral_equation}
u  = \bR(Q_{\lambda}|u|^{p-2}u),\qquad~~~~~~~ u \in L^{p}(\RN). 
\end{equation}
where $\bR(f) = \Psi \ast f$ for the function $\Psi$ introduced in~\eqref{eq:defnPsi}  and
\begin{equation}\label{eqn:resolvent}
\left\| \bR(f) \right\|_{L^{p}(\RN)} \leq C \left\|f\right\|_{L ^{p'}(\RN)}
\end{equation}
for all $p \in \left[\frac{2(N+1)}{N-1},\frac{2N}{N-2} \right]$ and some constant $C>0$. 

\medskip 

\noindent
To obtain the dual variational formualation of~\eqref{eqn:integral_equation} we introduce $v:=
\1_{A_{+}}u$ and $w:= \1_{A_{-}}u$. Then  \eqref{eqn:integral_equation} is equivalent to the system
\vspace{-6pt}
\begin{align*}
  \begin{aligned}
v &= \lambda \mathds{1}_{A_{+}}\bR\left[Q_+|v|^{p-2}v\right] -  \mathds{1}_{A_{+}}\bR\left[
Q_-|w|^{p-2}w \right], \\
w&=  \lambda \mathds{1}_{A_{-}}\bR\left[Q_+|v|^{p-2}v\right] -  \mathds{1}_{A_{-}}\bR\left[
  Q_-|w|^{p-2}w \right]. 
    \end{aligned}
\end{align*}
\vspace{-6pt}
Setting  
\begin{equation*}
  \phi := \lambda Q_+^{1/p'}|v|^{p-2}v \in L^{p'}(A_{+}),  \qquad 
  \psi:=   Q_-^{1/p'}|w|^{p-2}w \in L^{p'}(A_{-})
\end{equation*}
we deduce 
\vspace{-10pt}
\begin{align*}
  \lambda^{1-p'}|\phi|^{p'-2}\phi  
&= Q_+^{1/p} v  \\
&= \lambda Q_+^{1/p} \bR\left[Q_+|v|^{p-2}v\right] - Q_+^{1/p}
\bR\left[Q_-|w|^{p-2}w \right] \\
&=  Q_+^{1/p} \bR\left[Q_+^{1/p} \phi\right] -
 Q_+^{1/p} \bR\left[Q_-^{1/p} \psi\right] \\
 &=  Q_+^{1/p} \bR\left[|Q|^{1/p} (\phi-\psi)\right]. 
\end{align*}\vspace{-10pt}
\vspace{-10pt}
Similarly
\begin{align*}
 |\psi|^{p'-2}\psi 
&=  Q_-^{1/p } \bR\left[|Q|^{1/p}(\phi-\psi)\right].
 \end{align*}
In terms of the Birman-Schwinger operator $\cK : f\mapsto  |Q|^{\frac{1}{p}} \bR\big(|Q|^{\frac{1}{p}}f\big)$
introduced above this can be reformulated as \vspace{-10pt}
\begin{align*} 
  \lambda^{1-p'} |\phi|^{p'-2}\phi &= \mathds{1}_{A_{+}}\cK(\phi- \psi),  \\
   |\psi|^{p'-2}\psi &= \mathds{1}_{A_-}\cK(\phi- \psi) 
\end{align*}
and therefore carries a variational structure through the (dual) energy functional $J_\lambda$ on $L^{p'}(A_{+})\times L^{p'}(A_{-})$ given by
\begin{equation}\label{eqn:functional}
 J_{\lambda}(\phi,\psi):= \frac{\lambda^{1-p'}}{p'}\left\|\phi\right\|^{p'}_{p'} -
\frac{1}{p'}\left\|\psi\right\|^{p'}_{p'} - \frac{1}{2}\int\limits_{\RN}(\phi - \psi)\cK(\phi -
 \psi) ~dx.  
\end{equation}
This functional is of class $\cC^{1}$ with \vspace{-8pt}
\begin{align*}\label{eqn:derivative}
\partial_1J_{\lambda}(\phi,\psi)[h_{1}] &= \int\limits_{\RN} \left(\lambda^{1-p'}|\phi|^{p'-2}\phi -
\cK(\phi-\psi)\right)h_{1}~dx, &&h_{1} \in L^{p'}(A_{+}) \\
\partial_2 J_{\lambda}(\phi,\psi)[h_{2}] &=\int\limits_{\RN}\left(-|\psi|^{p'-2}\psi -
\cK(\phi-\psi)\right)h_{2}~dx, &&h_{2} \in L^{p'}(A_{-}).
\end{align*}
Here $\partial_1,\partial_2$ standard for partial derivatives with respect to $\phi$ and $\psi$.
For this reason we will look for critical points of $J_\lambda$. These solve the integral equation \eqref{eqn:integral_equation}. Thus by the regularity results \cite[Lemma 4.3]{Evequoz2015} and \cite[p.13]{Man_Uncountably} these are indeed strong solutions to our original problem \eqref{eqn:introeq}. 

\begin{proposition}\label{prop:crit_points_are solutions}
Let $(\phi,\psi) \in L^{p'}(A_{+})\times L^{p'}(A_{-})\setminus\{(0,0)\}$ be a critical point of $J_{\lambda}$
where $\lambda>0$. Then 
$$ 
  u  := \mathrm{{\bf R}} \left(|Q|^{\frac{1}{p}}(\phi-\psi)\right)\in L^{p}(\RN) 
$$
is a nontrivial solution of \eqref{eqn:integral_equation}. 
\end{proposition}

\begin{proof}
Let $(\phi,\psi) \in L^{p'}(A_{+})\times L^{p'}(A_{-})\setminus\{(0,0)\}$ be a critical point of $J_{\lambda}$. Thus we have
\begin{equation*}
 \lambda^{1-p'} |\phi|^{p'-2}\phi = \mathds{1}_{A_{+}}\cK(\phi- \psi),  \qquad   |\psi|^{p'-2}\psi =
 \mathds{1}_{A_-}\cK(\phi- \psi)
\end{equation*}
 as well as 
 \begin{align*}
  	Q_\lambda |u|^{p-2} u 
  	&= 
  	(\lambda \1_{A_+} - \1_{A_-}) |Q| |u|^{p-2} u 
  	\\
  	&=
  	(\lambda \1_{A_+} - \1_{A_-}) |Q|^\frac{1}{p} \cdot  ||Q|^\frac{1}{p} u|^{p-2} \cdot  |Q|^\frac{1}{p} u 
  	\\
  	&=
  	(\lambda \1_{A_+} - \1_{A_-}) |Q|^\frac{1}{p} \cdot  
  	||Q|^\frac{1}{p} \bR \left[|Q|^{1/p}(\phi-\psi) \right] |^{p-2} \cdot  |Q|^\frac{1}{p} \bR \left[|Q|^{1/p}(\phi-\psi) \right]
  	\\
  	&=
  	 (\lambda \1_{A_+} - \1_{A_-}) |Q|^\frac{1}{p} \cdot  
  	|\cK \left[\phi-\psi \right] |^{p-2} \cdot  \cK \left[\phi-\psi \right]
  	\\
  	&=
  	 (\lambda \1_{A_+} - \1_{A_-}) |Q|^\frac{1}{p} \cdot  
  	\left|	 \lambda^{1-p'} |\phi|^{p'-2}\phi + |\psi|^{p'-2}\psi \right|^{p-2} 
  	\cdot \left(    \lambda^{1-p'} |\phi|^{p'-2}\phi + |\psi|^{p'-2}\psi \right)
  	\\
  	&=
  	(\lambda \1_{A_+} - \1_{A_-}) |Q|^\frac{1}{p} \cdot  \left(   	 \lambda^{(1-p')(p-1)} \phi + \psi   	\right)
  	\\
  	&= 
  	|Q|^\frac{1}{p} \cdot  \left(  \phi - \psi   	\right).
  \end{align*}
  Applying $\bR$ then gives
  $
    \bR\left(Q_{\lambda}|u|^{p-2}u\right) = \bR\left(|Q|^{\frac{1}{p}}(\phi-\psi)\right) = u. 
  $
  Hence $u$ solves~\eqref{eqn:integral_equation}.
\end{proof}

  So we conclude that it remains to find nontrivial critical points of the functionals $J_\lambda$ for as many
  $\lambda>0$ as possible. This will be achieved with the Mountain Pass Theorem for families of
  $\cC^1$-functionals by Jeanjean and Toland~\cite{jeantoland98}.

\section{Saddle-point reduction}

In this section we perform the saddle-point reduction of $J_{\lambda}$ with respect to the $\psi$-variable. To
this end, we prove that for any fixed $\phi \in L^{p'}(A_{+})$ the functional $\psi\mapsto
J_{\lambda}(\phi,\psi)$ attains its maximum at some uniquely defined function in $L^{p'}(A_-)$ that we will call $Z(\phi)$ in the
following. 
We shall see that the positivity assumption $\int\limits_{\RN} \psi\cK\psi~dx \geq 0$ for all $\psi \in L^{p'}(A_{-})$ ensures that the functional $\psi\mapsto J_{\lambda}(\phi,\psi)$ is
strictly concave so that the global maximization with respect to $\psi$ is the only reasonable approach to
perform a saddle point reduction. We introduce the reduced functional $\tilde{J}_{\lambda}:L^{p'}(A_{+}) \to
\R$ via
\begin{equation}\label{eqn:reduced}
\tilde{J}_{\lambda}(\phi):= \sup\limits_{\psi \in L^{p'}(A_{-})}J_{\lambda}(\phi,\psi).
\end{equation}

\begin{proposition}\label{prop:sp_reduction}
Assume that $\cK : L^{p'}(\RN) \to L^{p}(\RN)$ is compact and that $ \int\limits_{A_-} \psi \cK \psi ~dx
\geq 0 $ for all $\psi\in L^{p'}(A_{-})$. Then for every $\phi \in L^{p'}(A_{+})$ there exists a unique
$Z(\phi) \in L^{p'}(A_{-})$ such that for all $\lambda > 0$ we have 
$$ 
  \tilde{J}_{\lambda}(\phi)
  = J_{\lambda}(\phi,Z(\phi)). 
$$ 
Moreover:
\begin{enumerate}[(i)]
\item For any $\phi \in L^{p'}(A_{+})$ the corresponding maximizer $Z(\phi)$ satisfies
\begin{equation} \label{eqn:maximizerbounded_prop} \left\|Z(\phi)\right\|_{p'} \leq \left(p'\beta \left\|\phi\right\|_{p'}\right)^{\frac{1}{p'-1}}\end{equation}
where $\beta$ is defined in \eqref{eqn:introalphabeta}.
\item  The map $Z : L^{p'}(A_{+}) \to L^{p'}(A_{-})$ is continuous.
\item The reduced functional $\tilde{J}_{\lambda} : L^{p'}(A_{+}) \to \R$ is of class $\cC^{1}$ with derivative
$$ \tilde{J}_{\lambda}'[h] = \partial_{1}J_{\lambda}(\phi,Z(\phi))[h]. $$
\end{enumerate}
\end{proposition}

\begin{proof}
We first establish the existence of a maximizer. 
So fix $\phi \in L^{p'}(A_{+})$ and consider a maximizing
sequence $(\psi_{n})_{n} \subset L^{p'}(A_{-})$. Using $J_{\lambda}(\phi,0) \leq \sup\limits_{\psi \in
L^{p'}(A_{-})}J_{\lambda}(\phi,\psi) = J_{\lambda}(\phi,\psi_{n})+ o(1)$ as $n \to \infty$  we obtain 
$$ o(1) \leq -\frac{1}{p'}\|\psi_n\|_{p'}^{p'} + \int_{\RN} \phi\mathcal K\psi_n - \frac{1}{2} \int_{\RN}
  \psi_n \mathcal K\psi_n
  \leq -\frac{1}{p'}\|\psi_n\|_{p'}^{p'} + \beta \|\phi\|_{p'} \|\psi_n\|_{p'}   \quad (n \to \infty).
$$
Here we used the nonnegativity assumption on $\cK$ as well as~\eqref{eqn:resolvent}. Hence,
\begin{equation}\label{eqn:maximizer_bounded}
  \|\psi_n\|_{p'} \leq (p'\beta\|\phi\|)^{\frac{1}{p'-1}} + o(1)  \qquad (n \to \infty),
\end{equation}
so $(\psi_{n})_{n}$ is bounded. Passing to a subsequence we find $\psi^{\ast} \in L^{p'}(A_{-})$ such that
$\psi_{n} \weakto \psi^{\ast}$ in $L^{p'}(A_{-})$ as $n \to \infty.$ Using the compactness of $\cK$ and
the weak lower semicontinuity of the norm we find
\begin{align*}
     &\sup_{\psi \in L^{p'}(A_{-})} J_\lambda(\phi,\psi) 
 		\\    
     &\quad 
     =    \frac{\lambda^{1-p'}}{p'} \|\phi\|_{p'}^{p'} 
     - \frac{1}{p'} \|\psi_n\|_{p'}^{p'}
        - \frac{1}{2}   \int_{\RN} \phi\mathcal K \phi ~dx
        +  \int_{\RN} \phi\mathcal K\psi_n ~dx
        - \frac{1}{2} \int_{\RN}  \psi_n\mathcal K\psi_n ~dx  + o(1)
         \\
     &\quad
     =  \frac{\lambda^{1-p'}}{p'} \|\phi\|_{p'}^{p'}
      - \frac{1}{p'} \|\psi_n\|_{p'}^{p'}
        - \frac{1}{2}   \int_{\RN} \phi\mathcal K \phi ~dx
        +  \int_{\RN} \phi\mathcal K\psi^\ast ~dx
        - \frac{1}{2} \int_{\RN}  \psi^\ast \mathcal K\psi^\ast ~dx  + o(1)
         \\
     &\quad  
     \leq \frac{\lambda^{1-p'}}{p'} \|\phi\|_{p'}^{p'}   
     - \frac{1}{p'} \|\psi^\ast\|_{p'}^{p'}
        - \frac{1}{2}   \int_{\RN} \phi\mathcal K \phi ~dx
        +  \int_{\RN} \phi\mathcal K\psi^\ast ~dx
        - \frac{1}{2} \int_{\RN}  \psi^\ast \mathcal K\psi^\ast ~dx  + o(1)
         \\
    &\quad
    = J_\lambda (\phi,\psi^\ast) + o(1).
\end{align*}
Hence the supremum is attained at $\psi^{\ast}$. Since equality must hold in the above estimate we conclude
$\left\|\psi_{n}\right\|_{p'} \to \left\|\psi^{\ast}\right\|_{p'}$, whence $\psi_{n} \to \psi^{\ast}$ in
$L^{p'}(A_{-})$ as $n \to \infty$.  This shows the existence of a maximizer satisfying the estimate stated
in~(i). So (i) is proved once we have established the uniqueness of the maximizer.

\medskip

To this end assume that $\psi^\ast, \psi^\dagger \in L^{p'}(A_-)$ are maximizers. Then we have
\begin{align*}
	0 
	&\leq
	\frac{1}{2} J_\lambda (\phi, \psi^\ast) + \frac{1}{2}  J_\lambda (\phi, \psi^\dagger) 
	-  J_\lambda \left(\phi, \frac{1}{2}(\psi^\ast + \psi^\dagger)\right) 
	\\
	&=
	\frac{1}{p'}\left(\left\|\frac{\psi^\ast + \psi^\dagger}{2}\right\|^{p'}_{p'} 
	- \frac{1}{2} \|\psi^\ast \|^{p'}_{p'} - \frac{1}{2} \| \psi^\dagger\|^{p'}_{p'}  \right)
	\\
	& \quad \quad
	+ \frac{1}{2} \left( 
	 \int_{\RN} \frac{\psi^\ast + \psi^\dagger}{2} 
	 \mathcal{K}\left[\frac{\psi^\ast + \psi^\dagger}{2}\right] \dx
	- \frac{1}{2} \int_{\RN} \psi^\ast \mathcal{K} \psi^\ast \dx 
	- \frac{1}{2} \int_{\RN} \psi^\dagger \mathcal{K} \psi^\dagger \dx  \right)
	\\
	&=
	\frac{1}{p'}\left( \left\|\frac{\psi^\ast + \psi^\dagger}{2}\right\|^{p'}_{p'}  -  \frac{1}{2} \|\psi^\ast
	\|^{p'}_{p'} - \frac{1}{2} \| \psi^\dagger\|^{p'}_{p'} \right)
	- \frac{1}{8}  
	\int_{\RN} (\psi^\ast - \psi^\dagger) \mathcal{K} [\psi^\ast - \psi^\dagger] \dx \\
	&\leq 
	\frac{1}{p'}\left( \left\|\frac{\psi^\ast + \psi^\dagger}{2}\right\|^{p'}_{p'}  -  \frac{1}{2} \|\psi^\ast
	\|^{p'}_{p'} - \frac{1}{2} \| \psi^\dagger\|^{p'}_{p'} \right) \\
	&\leq 0,  
\end{align*}
where we have used the non-negativity condition in the second last step and the convexity of $z\mapsto |z|^{p'}$
in the last step. So we have equality in each estimate and conclude $\psi^\ast = \psi^\dagger$. 
Note that the maximizer does not depend on $\lambda $ since the only $\lambda $-dependent term in
$J_\lambda(\phi,\psi)$ is  $\frac{\lambda^{1-p'}}{p'}\|\phi\|_{p'}^{p'}$, which is independent of $\psi$. 

\medskip

We now prove (ii), i.e., the continuity of the map $Z$: Assume $\phi_{n} \to \phi$ in $L^{p'}(A_{+})$ and  let
$(\psi_{n})_{n} := (Z(\phi_n))_n \subset L^{p'}(A_{-})$ be the associated maximizers. By \eqref{eqn:maximizer_bounded}, the
sequence $(\psi_{n})_{n}$ is bounded and after passing to a subsequence we may assume $\psi_{n} \weakto \psi_{0}$ in
$L^{p'}(A_{-})$ as $n \to \infty$. Arguing as above we deduce
\begin{equation}\label{eqn:uppermax}
     \limsup_{n\to\infty} \tilde J_\lambda (\phi_n)
     =  \limsup_{n\to\infty} J_\lambda (\phi_n,\psi_n)   
     =  \limsup_{n\to\infty} J_\lambda (\phi,\psi_n) 
     \leq J_\lambda (\phi,\psi_{0}) 
     \leq \tilde J_\lambda (\phi).
\end{equation} 
using weak lower semicontinuity and $\liminf\limits_{n \to \infty} \left\|\psi_{n}\right\|_{p'} \geq
\left\|\psi_{0}\right\|_{p'}$.  On the other hand, with the special choice $\psi = Z(\phi)$ we obtain
\begin{equation} \label{eqn:lowermax}
      \liminf_{n\to\infty} \tilde J_\lambda (\phi_n)
     \geq J_\lambda (\phi,\psi)  
     = J_\lambda (\phi,Z(\phi))
     = \tilde J_\lambda (\phi).
\end{equation}
Combining both estimates gives $\left\|\psi_{n}\right\|_{p'} \to \left\|\psi_{0}\right\|_{p'}$ as well as
$\tilde{J}_{\lambda }(\phi_{n}) \to \tilde{J}_{\lambda }(\phi)$ as $n \to \infty$. Thus we have equality in
\eqref{eqn:uppermax}, \eqref{eqn:lowermax}. Since maximizers are unique, we obtain $\psi_{0} = Z(\phi)$ and
in particular $Z(\phi_n)=\psi_{n} \to \psi_{0}=Z(\phi_0)$ in $L^{p'}(A_{-})$ as $n \to \infty$. 

\medskip

We are left to prove (iii). Let $h \in L^{p'}(A_{+})$ be arbitrary. We can estimate the difference quotients as follows:
\begin{align*}
    \liminf_{\tau\to 0} \frac{\tilde J_\lambda (\phi+\tau h)-\tilde J_\lambda (\phi)}{\tau}
    &\geq \liminf_{\tau\to 0} \frac{J_\lambda (\phi+\tau h,Z(\phi))-J_\lambda (\phi,Z(\phi))}{\tau} \\
    &=  \liminf_{\tau\to 0} \int_0^1 
    \partial_1 J_\lambda (\phi+\tau \sigma h,Z(\phi))[h] \,\mathrm{d}\sigma \\  
    &= \partial_1 J_\lambda (\phi,Z(\phi))[h], \\ 
    \limsup_{\tau\to 0} \frac{\tilde J_\lambda (\phi+\tau h)-\tilde J_\lambda (\phi)}{\tau}
    &\leq \limsup_{\tau\to 0} \frac{J_\lambda (\phi+\tau h ,Z(\phi+\tau h))-
    J_\lambda (\tilde\phi,Z(\phi+\tau h))}{\tau} \\
    &= \limsup_{\tau\to 0}  \int_0^1 \partial_1
    J_\lambda (\phi^*+\tau \sigma h,Z(\phi+\tau h))[h]\,\mathrm{d}\sigma \\
    &= \partial_1 J_\lambda (\phi,Z(\phi))[h].  
\end{align*} 
Here we used that $Z$ is continuous and that $\partial_1 J_\lambda $ is continuous,
see~\cite[Proposition~9]{SzuWeth_Nehari} for a similar computation. We
conclude that $\tilde J_\lambda $ is  G\^{a}teaux-differentiable with continuous derivative $\phi \mapsto
\partial_1 J_\lambda (\phi,Z(\phi))[\cdot]$, see Proposition~\ref{prop:sp_reduction}~(iii). 
Hence, the reduced functional $\tilde J_\lambda $ is continuously (Fr\'{e}chet-)differentiable with
\begin{align*}
    \tilde J_\lambda '(\phi)[h]
    = \partial_1 J_\lambda (\phi,Z(\phi))[h]
    \qquad\forall h\in L^{p'}(A_+)
\end{align*} 
as claimed.  
\end{proof}

Notice that the condition $\int_{\R^N} \psi\mathcal K\psi\dx \geq 0$ is also necessary for the existence of
a global maximizer of $\psi\mapsto J_\lambda(\phi,\psi)$ because otherwise this functional is unbounded from above.

\section{Palais-Smale sequences for the reduced functional}

In view of the results of the previous sections, we obtain a solution  to our problem by
proving the existence of a nontrivial critical point  of the reduced functional $\tilde{J}_{\lambda } :
L^{p'}(A_{+}) \to \R$ introduced in~\eqref{eqn:reduced}. This will be done via Mountain-pass techniques for monotone
families of functionals originating from the work of Jeanjean and Toland~\cite{jeantoland98}.

\begin{defi} \label{def:MPGeometry}
    Let $X$ be a Banach space, $M\subset\R$ a compact interval. Then the family $(I_\nu)_{\nu\in
    M}$ of $\cC^1$-functionals on $X$ is said to have the Mountain Pass Geometry if there exist $v_1,v_2\in X$
    such that for all $\nu\in M$ it holds
    $$
      c_\nu := \inf_{\gamma\in \Gamma} \sup_{t\in [0,1]} I_\nu(\gamma(t)) 
      >\max\{I_\nu(v_1),I_\nu(v_2)\},
    $$
    where $\Gamma:=\{\gamma\in C([0,1],X): \gamma(0)=v_1,\gamma(1)=v_2\}$.
  \end{defi}  
  
\begin{theorem}{(Jeanjean, Toland) \cite[Theorem 2.1]{jeantoland98}}\label{thm:JeanjeanToland}
    Assume that $X$ is a Banach space, $M\subset\R$ a compact interval and 
     $(I_\nu)_{\nu \in  M}$ a family of $\cC^1$-functionals on $X$ having the Mountain Pass Geometry. 	 Assume
     further that $(I_\nu)_{\nu \in  M}$ has the following property:
	\begin{align}\label{eqn:conditionH}\tag{\textbf{H}}
	\begin{aligned}
			&\text{For every sequence } (\nu_n,\phi_n)\in M \times X \text{ with } 
			\nu_n\nearrow \nu_*\in M \text{ and with } 
			\\
			& \quad 
			-I_{\nu_*}(\phi_n),
			\quad I_{\nu_n}(\phi_n),
			\quad \frac{I_{\nu_n}(\phi_n)-I_{\nu_*}(\phi_n)}{\nu_*-\nu_n}
     	  	\quad\text{ bounded from above, }
     	  	\\
     	  	& \text{the sequence  } (\phi_n) \text{ is bounded itself, and } 
     	  	\limsup_{n\to\infty} (I_{\nu_*}(\phi_n)-I_{\nu_n}(\phi_n))\leq 0.
	\end{aligned}
	\end{align}	     
     Then for almost all $\nu\in M$ there is a bounded Palais-Smale sequence (BPS) for $I_\nu$ at the level
     $c_{\nu}$.
\end{theorem}

We shall apply this result to $X = L^{p'}(A_{+})$ and the family of $\cC^{1}$-functionals
$I_\lambda:= \tilde{J}_\lambda:X\to\R$. We first verify the Mountain Pass Geometry for
parameters $\lambda\in (\lambda_0,\infty)$ where $\lambda_0= (2\beta\alpha^{-1})^p$. 
Let us recall that $\alpha,\beta$ were defined as
$$ 
 \alpha:= \max\limits_{\left\|\phi\right\|_{p'}  = 1, \atop \supp(\phi)\subset A_+}\int\limits_{\RN} \phi
  \cK \phi ~dx, \qquad \beta:= \max\limits_{\left\|\phi\right\|_{p'} 
  = \left\|\psi\right\|_{p'} =1, \atop \supp(\phi)\subset A_+,\supp(\psi)\subset A_-} \int\limits_{\RN}
  \phi \cK \psi ~dx.
$$

%

 \begin{proposition}\label{prop:MP}
Let  $\cK : L^{p'}(\RN) \to L^{p}(\RN)$ be compact and assume $\int\limits_{\RN}\psi \cK \psi ~dx \geq 0$ for
all $\psi \in L^{p'}(A_{-})$. Then, for any given compact subinterval $M\subset (\lambda_0,\infty)$, the
family of functionals $(\tilde{J}_{\lambda})_{\lambda \in M}$ has the Mountain Pass Geometry according to
Definition~\ref{def:MPGeometry}.
\end{proposition}
\begin{proof}
  For $\lambda\in M$ we define $r_\lambda:= (\lambda^{p'-1}\alpha)^{1/(p'-2)}$. Then we have 

\begin{align*}
  \inf_{\|\phi\|_{p'}= r_\lambda} \tilde J_\lambda (\phi)
  &= \inf_{\|\phi\|_{p'}=  r_\lambda}  \sup_{\psi \in L^{p'}(A_{-})} J_\lambda(\phi,\psi) 
  \geq \inf_{\|\phi\|_{p'}= r_\lambda}  J_\lambda (\phi, 0)  \\
  &= \inf_{\|\phi\|_{p'}=  r_\lambda}  \frac{\lambda^{1-p'}}{p'} \|\phi\|_{p'}^{p'} - \frac{1}{2} \int_{\RN}
  \phi \mathcal{K} \phi \dx\\
  &=\frac{\lambda^{1-p'} }{p'} r_\lambda^{p'} - \frac{ \alpha}{2}  r_\lambda^2 
  = \alpha\left(\frac{1}{p'}-\frac{1}{2}\right) (\lambda^{p'-1}\alpha)^{\frac{2}{p'-2}} \\ 
   &>0.
\end{align*}
On the other hand,  we have
\begin{align*}
	\tilde J_\lambda (0) 
	= \sup_{\psi \in L^{p'}(A_{-})} J_\lambda(0,\psi) 
	= \sup_{\psi \in L^{p'}(A_{-})} \left[ -\frac{1}{p'} \|\psi\|_{p'}^{p'} - \frac{1}{2} \int_{\RN} \psi \mathcal{K}\psi \dx\right]
	= 0.
 \end{align*}
 According to Definition~\ref{def:MPGeometry} it therefore remains to find some
 $\phi^* \in L^{p'}(A_+)$ with $\|\phi^*\|_{p'}\geq r_{\lambda_0}$ such that
 $\tilde J_{\lambda_0}(\phi^*)\leq 0$ holds. Notice that in this case we actually have $\tilde
 J_{\lambda}(\phi^*)<\tilde J_{\lambda_0}(\phi^*)\leq 0$ for all $\lambda\in M\subset(\lambda_{0},\infty)$.
To achieve this we estimate $\tilde J_{\lambda_0}$ from above as follows 
  \begin{align*}
    \tilde J_{\lambda_0}(\phi)
    &= J_{\lambda_0}(\phi,0) - \frac{1}{p'}\|Z(\phi)\|_{p'}^{p'} + \int_{\RN} \phi\cK (Z(\phi)) -
    \frac{1}{2}\int_{\RN} Z(\phi)\mathcal K(Z(\phi)) \\
    &\leq J_{\lambda_0}(\phi,0) - \frac{1}{p'}\|Z(\phi)\|_{p'}^{p'} + \beta \|\phi\|_{p'}\|Z(\phi)\|_{p'} \\   
    &\leq J_{\lambda_0}(\phi,0) + \frac{\beta^p}{p}\|\phi\|_{p'}^{p}   
  \end{align*}
  where we have used that $ \max\limits_{c\geq 0}\left( -\frac{c^{p'}}{p'} + \beta\left\|\phi\right\|^{p'}_{p'}c\right) = \frac{\beta^p}{p}\|\phi\|_{p'}^{p} .$  We choose $\phi^*=  r_0 \phi_0$ where the function $\phi_0 \in
  L^{p'}(A_+), \|\phi_0\|_{p'} = 1$ attains the maximum $\alpha = \int_{\RN} \phi_0 \mathcal{K} \phi_0 \dx >
  0$. Then the choice $R:= (\frac{1}{2}\alpha\beta^{-p})^{1/(p-2)}$ yields after some computations 
  (recall $\lambda_0= (2\beta\alpha^{-1})^p$)
  $$
    \|\phi^*\|_{p'} 
    = R
    = \left(\frac{1}{2}\alpha\beta^{-p}\right)^{\frac{1}{p-2}} 
    \geq (\lambda_0^{p'-1}\alpha)^{\frac{1}{p'-2}}
    > (\lambda^{p'-1}\alpha)^{\frac{1}{p'-2}}
    = r_\lambda
    \qquad\text{for all }\lambda\in M.
  $$
  Using again the explicit formulas for $R,\lambda_0$ we find   
\begin{align*}
 \tilde{J}_{\lambda_0}(R\phi_0)
  & \leq J_{\lambda_0}(R\phi_0)  + \frac{\beta^p R^p}{p}  \\
  &=  \frac{\lambda_0^{1-p'}}{p'} R^{p'}- \frac{\alpha}{2} R^2  + \frac{\beta^p }{p}R^p  \\
  &=  \frac{R^{p'}}{p'} \cdot \left( \lambda_0^{1-p'}   - \frac{p\alpha}{2(p-1)} R^{2-p'}  +
  \frac{\beta^p }{p-1} R^{p-p'}\right) \\
  &=  \frac{R^{p'}}{p'}\cdot \left(  \lambda_0^{1-p'}   - \frac{1}{2} \alpha R^{2-p'}   \right) \\
  &=  \frac{R^{p'}}{p'}\cdot \left(  (2\beta\alpha^{-1})^{-\frac{p}{p-1}}   - \frac{1}{2} \alpha\cdot
  \left(\frac{1}{2}\alpha\beta^{-p}\right)^{\frac{1}{p-1}}  \right) \\
  &=  0 
%
\end{align*}
and thus the claim holds with $v_{1} = 0$ and $v_{2} = \phi^*=R\phi_0$.
\end{proof}

Having established the Mountain Pass Geometry of our functionals we now verify the
condition~\eqref{eqn:conditionH} in order to use Theorem~\ref{thm:JeanjeanToland}

\begin{proposition}\label{prop:ConditionH}
  For any compact subinterval $M\subset (\lambda_0,\infty)$ the family of $\cC^{1}-$functionals
  $(\tilde{J}_{\lambda})_{\lambda\in M}$ satisfies the condition~\eqref{eqn:conditionH}.
\end{proposition}
\begin{proof}
  Consider a sequence $(\phi_n, \lambda_n) \in L^{p'}(A_+) \times M$ with $\lambda_n \nearrow
  \lambda_*$ and
  \begin{align*}
   -\tilde{J}_{\lambda_*}(\phi_n) \leq C, 
   \quad
   \tilde{J}_{\lambda_n}(\phi_n) \leq C,
   \quad
   \frac{\tilde{J}_{\lambda_n}(\phi_n)-\tilde{J}_{\lambda_*}(\phi_n)}{\nu_*-\nu_n} \leq C
 \end{align*}    
 for all $n \in \N$. Then we have 
  $$ 
    C
    \geq \frac{\tilde{J}_{\lambda_n}(\phi_n)-\tilde{J}_{\lambda^*}(\phi_n)}{\lambda_*-\lambda_n}
    = \frac{\lambda_n^{1-p'}  - \lambda_*^{1-p'}}{\lambda_*-\lambda_n} \|\phi_n\|_{p'}^{p'}
    = \left( (p'-1)\lambda_*^{-p'}+o(1)\right) \|\phi_n\|_{p'}^{p'}
    \qquad (n\to\infty)
  $$
  So we conclude that  $(\phi_n)$ is bounded. Furthermore, $\lambda_n\to\lambda^*>0$ gives
  $$
    \limsup_{n\to\infty}\, (I_{\lambda_*}(\phi_n)-I_{\lambda_n}(\phi_n))
    = \limsup_{n\to\infty}\, (\lambda_n^{1-p'}-\lambda_*^{1-p'}) \|\phi_n\|^{p'}_{p'}
    = 0,
  $$ 
  which is all we had to show.
\end{proof}

 We thus conclude that Theorem~\ref{thm:JeanjeanToland} applies in our context and yields BPS sequences for
 $\tilde J_\lambda$ at the corresponding Mountain pass levels $c_\lambda$ for almost all $\lambda \in
 (\lambda_0,\infty)$. From the existence of BPS sequences we deduce rather easily the existence of critical
 points at the corresponding Mountain Pass level.
  
\begin{proposition} \label{prop:PScond}
  Let  $\cK : L^{p'}(\RN) \to L^{p}(\RN)$ be compact and assume $\int\limits_{\RN}\psi \cK \psi ~dx \geq 0$ for
  all $\psi \in L^{p'}(A_{-})$. Then for all $\lambda\in
  (\lambda_0,\infty)$ every BPS sequence of   $\tilde J_\lambda$ at its Mountain Pass level $c_\lambda$
  converges to a critical point of $\tilde J_\lambda$ at the level $c_\lambda$.
\end{proposition}
\begin{proof}
  Let $(\phi_j)_j$ in $L^{p'}(A_+)$ be a BPS sequence for $\tilde J_\lambda$, i.e.,  
  $\tilde{J}_\lambda(\phi_j) \to c > 0$ and $\tilde{J}_\lambda'(\phi_j) \to 0$.  
  We may thus assume w.l.o.g. $\phi_j \rightharpoonup \phi^*$. Moreover,
  Proposition~\ref{prop:sp_reduction}~(i) implies the
  boundedness of $(\psi_j)_j := (Z(\phi_j))_j$ and hence w.l.o.g. also weak convergence. For all $h \in L^{p'}(A_+)$ we then
  have, in view of the formula for $\tilde J_\lambda'$ from Proposition~\ref{prop:sp_reduction}~(iii),
\begin{align*}
	&
	\left|
	\int_{\RN} |\phi_j|^{p'-2}\phi_j h - |\phi_k|^{p'-2} \phi_k h \dx
	\right|
	\\
	& \quad 
	=  
	\left|
	\tilde{J}_\lambda'(\phi_j) h - \tilde{J}'(\phi_k) h
	+ \int_{\RN} h \mathcal{K}[\phi_j - \phi_k] \dx
	- \int_{\RN} h \mathcal{K}[\psi_j - \psi_k] \dx
	\right|
	\\
	& \quad
	\leq  
	\|h\|_{p'} \cdot \left[ \|\tilde{J}_\lambda'(\phi_j)\| + \|\tilde{J}_\lambda'(\phi_k)\| + 
	\|\mathcal{K}[\phi_j - \phi_k]\|_{p}
	+ \|\mathcal{K}[\psi_j - \psi_k]\|_{p}  \right]\\
	& \quad 
	= \|h\|_{p'} \cdot o(1) \qquad (j,k\to\infty).
\end{align*}
We infer that $(|\phi_j|^{p'-2}\phi_j)_j$ converges strongly in $L^p(A_+)$. By
uniqueness of weak limits, we infer $|\phi_j|^{p'-2}\phi_j \to |\phi^*|^{p'-2}\phi^*$
strongly in $L^p(A_+)$ and hence in particular $\|\phi_j\|_{p'} \to \|\phi^*\|_{p'}$. This finally
implies $\phi_j \to \phi^*$ strongly in $L^{p'}(A_+)$. A standard computation finally shows $\tilde
J_\lambda(\phi^*)=c_\lambda$ as well as $\tilde J_\lambda'(\phi^*)=0$.
\end{proof}

\section{Proof of Theorem~\ref{thm:introthm} and Corollary~\ref{cor}}

We finally combine all auxiliary results to prove Theorem~\ref{thm:introthm}.

\begin{proof}[Proof of Theorem~\ref{thm:introthm}]
From Proposition~\ref{prop:crit_points_are solutions} and Proposition~\ref{prop:sp_reduction} we infer that 
for almost all $\lambda\in (\lambda_0,\infty)$  a nontrivial solution $u\in L^p(\RN)$ of the nonlinear
Helmholtz equation~\eqref{eqn:introeq} is found once we have proved the existence of nontrivial critical points of the
reduced functional $\tilde J_\lambda$ for almost all $\lambda\in M$ where $M$ is an arbitrary compact
subinterval of $(\lambda_0,\infty)$.
From Proposition~\ref{prop:MP}  we infer that the family  $(\tilde J_\lambda)_{\lambda\in M}$ has the Mountain
Pass Geometry. Moreover, by Proposition~\ref{prop:ConditionH}, condition~\eqref{eqn:conditionH} holds.
So Theorem~\ref{thm:JeanjeanToland} yields for almost all $\lambda\in M$ a BPS sequence for 
$\tilde J_\lambda$ at the corresponding Mountain Pass level. By Proposition~\ref{prop:PScond} each of these
BPS sequences converges to a critical point $\phi_\lambda$ of $\tilde J_\lambda$ at the Mountain Pass level. 
Since this critical point is necessarily nontrivial, we have thus obtained the desired claim for
$L^p(\RN)$-solutions of~\eqref{eqn:introeq}. From \cite[Lemma 4.3]{Evequoz2015} we infer that
each of these solutions belongs to $W^{2,q}(\RN)\cap \cC^{1,\alpha}(\RN)$ for all $p \leq q < \infty$ and
$\alpha \in (0,1)$. Arguing as in Step~3 and Step~4 \cite[p.13]{Man_Uncountably} one even obtains that these
solutions belong to $W^{2,q}(\RN)$ for all $q\in (\frac{2N}{N-1},p)$. 
In particular, these solutions are strong solutions of \eqref{eqn:introeq}, which
finishes the proof.  
\end{proof}

\begin{proof}[Proof of Collorary~\ref{cor}]
In order to apply Theorem~\ref{thm:introthm} we show that~\eqref{eq:SufficientConditions}
implies~\eqref{eq:abstractassumptions}. In the special case $k=1$ the compactness of $\cK$ was shown in
Lemma~4.2 in~\cite{Evequoz2015}. So the general case follows by rescaling. It therefore remains to show that
$\delta:=\diam{(A_{-})} \leq k^{-1}y_{\frac{N-2}{2}}$ implies $\int\limits_{\RN}\psi \cK \psi~dx \geq 0$ for all $\psi \in L^{p'}(A_{-})$. 
Due to \eqref{eqn:resolvent}, $\mathcal K=|Q|^{1/p}\mathcal \bR(|Q|^{1/p}\cdot)$ and $Q \in
L^{\infty}(\RN)$ it suffices to prove
\begin{equation}\label{eqn:resolventpositive}
  \int\limits_{\RN} \psi \bR \psi~dx \geq 0, \qquad \text{ for all } \psi \in \SR(A_{-}).
\end{equation}
Using  that $x,y\in A_-$ implies $x-y\in B_\delta$ we infer from Corollary~5.4 in~\cite{chen2019complex} 
$$ 
\int\limits_{\RN} \psi \bR  \psi~dx 
= \int\limits_{\RN} \psi [\1_{B_{\delta}}\Psi  \ast \psi](x)~dx 
\geq 0
$$
which proves \eqref{eqn:resolventpositive} and hence the Corollary.  

\end{proof}

\bibliographystyle{abbrv}	
\bibliography{biblio}

\end{document}